\documentclass[11pt]{amsart}
\usepackage{amssymb}
\newtheorem{thm}{Theorem}[section]
\newtheorem{cor}[thm]{Corollary}
\newtheorem{lem}[thm]{Lemma}

\newtheorem{prop}[thm]{Proposition}
\theoremstyle{definition}

\theoremstyle{remark}
\newtheorem{rem}[thm]{Remark}
\numberwithin{equation}{section}

\newcommand{\R}{\mathbb R }

\newcommand{\C}{\mathbb C }

\newcommand{\pa}{\partial }
\newcommand{\La}{\langle}
\newcommand{\Ra}{\rangle}

\newcommand{\Hc}{\mathcal H}
\newcommand{\sd}{{\mathbf S}^{d-1}}
\newcommand{\std}{{ \mathbf S}^{2d-1}}
\newcommand{\om}{ \omega}
\newcommand{\D}{\Delta}
\newcommand{\ap}{\alpha}
\newcommand{\bt}{\beta}
\newcommand{\N}{\nabla}

\begin{document}

\title[]{Revisiting Riesz transforms for Hermite and Special Hermite Operators}
\author{Pradeep Boggarapu}
\author{S. Thangavelu}

\address{department of mathematics,
Indian Institute of Science, Bangalore - 560 012, India}
\email{pradeep@math.iisc.ernet.in}
\email{veluma@math.iisc.ernet.in}


\begin{abstract}
In this paper we prove weighted mixed norm estimates for Riesz
transforms associated to Hermite and special Hermite operators.
The estimates are shown to be equivalent to vectorvalued esimates for a sequence of operators defined in terms of Laguerre functions of different type. 

\end{abstract}

\vspace{3mm}



\maketitle

\section{Introduction}
\setcounter{equation}{0} The boundedness of Riesz transforms
associated to Hermite and special Hermite operators have been
studied by several authors, see \cite{ST} and \cite{STR}. Here in
this article we are interested in mixed norm estimates. To fix the
notation let $H = -\triangle + |x|^{2} $ be the Hermite operator
on $\R^d$ which can be written as $H= \frac{1}{2}\sum _{j=1} ^{d}
(A_jA_j ^*+ A^*_jA_j)$, where $A_j= \frac{\partial}{\partial x_j}
+ x_j$ and $A_j ^* = - \frac{\partial}{\partial x_j} + x_j$ are
the annihilation and creation operators. The Riesz transforms
associated to the Hermite operator are defined by $R_j=A_j
H^{-\frac{1}{2}}$ and $R_j^*=A^*_j H^{-\frac{1}{2}}$, $j=1,2,
\cdots , d$. It is well known that they are bounded on
$L^p(\R^d)$, $1< p < \infty$ and
weak type $(1, 1)$. \\

Given a weight function $w(r)$ defined on $\mathbb{R}^+=(0, \infty
)$ we consider the mixed norm space $L^{p,2} (\R^d, w )$ which
consists of all measurable functions $f$ on $\R^d =
\mathbb{R}^+\times \mathbf{S}^{d-1}$ for which
$$\| f \| ^p _{L^{p,2}(\R^d, w)} = \int^{\infty} _0 \Big ( \int_{\sd} |f(r\omega)|^2 d\omega \Big )^{\frac{p}{2}} w(r) r^{d-1}dr $$
 are finite. One of our main results in this paper is the
 following.

 \begin{thm}
 For $j=1,2, \ldots , d $ we have
 $$\| R_j f \| _{L^{p,2}(\R^d, w)}\leq C \| f \|_{L^{p,2}(\R^d, w)}$$
 for all $1<p<\infty $, and $w \in A^{\frac{d}{2}-1}_p(\mathbb{R}^+)$. Similar estimates are valid
 for $R^*_j$ also.
 \end{thm}

 In the above theorem $A^{\alpha}_p(\mathbb{R}^+)$ is the Muckenhoupt's $A_p$ -
 class defined on $\mathbb{R}^+$ for $\alpha \geq - \frac{1}{2} $ with respect to the measure $d\mu _{\alpha}(r) =
 r^{2 \alpha +1}dr$. The above result leads to a vector valued inequality
 for a sequence of Laguerre Riesz transforms, see Theorem 2.6.\\

 We also consider the Riesz transforms $S_j$ and $\overline{S}_j$
 associated to the special Hermite operator $L= - 2 \sum _{j=1} ^{d} (Z_j\overline{Z}_j + \overline{Z}_j Z_j ) $
 where $Z_j = \frac{\partial}{\partial z_j} + \frac{1}{4}
 \overline{z}_j$ and $\overline{Z}_j = \frac{\partial}{\partial \overline{z}_j} - \frac{1}{4}
 z_j$. Again using spectral theorem we define $S_j = Z_j
 L^{- \frac{1}{2}}$ and $\overline{S}_j = \overline{Z}_j L^{- \frac{1}{2}}$. For
 these operators we prove

 \begin{thm}

 For $j=1,2, \ldots , d $ we have
 $$\| S_j f \| _{L^{p,2}(\C^d, w)}\leq C \| f \|_{L^{p,2}(\C^d, w)}$$
 for all $1<p<\infty $, and $w \in A^{d-1}_p(\mathbb{R}^+)$.
 Similar estimates are valid for $\overline{S}_j$ also.
 \end{thm}

 In this article we give three different proofs of Theorem 1.1 .
In a forthcoming paper \cite{BT} we prove mixed norm estimates for
Riesz transforms associated to Dunkl harmonic oscillator. Of the
three proofs mentioned above, only one proof can be modified to
treat the case of Dunkl harmonic oscillator.

\section{Riesz transforms for the Hermite Operator}

\subsection{Mixed norm estimates for Hermite Riesz transforms} We first
recall some results on Hermite expansions required for the proof
of Theorem 1.1. If $e^{-tH}$ is the Hermite semigroup generated by
$H$, then it is well known that
$$e^{-tH}f(x) = \int_{\R^d} K_t(x,y)f(y)dy $$
where the kernel is explicitly given by
$$K_t(x,y)= (2\pi)^{-\frac{d}{2}}(\sinh 2t )^{-\frac{d}{2}} e^{-\frac{1}{4}(\coth t)|x-y|^2}
e^{-\frac{1}{4}(\tanh t)|x+y|^2}.$$ The operator
$H^{-\frac{1}{2}}$ is then defined by
$$H^{-\frac{1}{2}} f(x) = \frac{1}{\sqrt{\pi}} \int _0 ^\infty e^{-tH}f(x) t^{-\frac{1}{2}}dt$$ and consequently the Riesz transforms $R_j$ are given by

$$R_jf(x) = \int_{\R^d} R_j(x,y)f(y)dy $$
where

\begin{eqnarray}
  R_j(x,y)= \frac{1}{\sqrt{\pi}} \int _0 ^ \infty \Big (
\frac{\partial}{\partial x_j} +x_j \Big ) K_t(x,y)
t^{-\frac{1}{2}} dt.
\end{eqnarray}
The following estimates on the kernel $R_j$ are known (see Theorem
3.3 in \cite{STR}).

\begin{prop}

For each $j=1,2, \ldots , d$; $x,y \in \R^d$ and $x\neq y$, there exist constants $C_j$ and $C'_j$ such that \\

 (1)  $ |R_j(x,y)| \leq C_j |x-y|^{-d} $\\

 (2) $| \nabla _x R_j(x,y)| \leq C'_j |x-y|^{-d-1}$.\\

\end{prop}

 As $R_j$ are clearly bounded on $L^2(\R^d)$, the above estimates
 lead to the boundedness of $R_j$ on $L^p(\R^d)$, $1<p<\infty $
 and also the weak type $(1,1)$ estimate. As we are interested in
 mixed norm estimates, we consider $L^{p ,2}(\R^d)$ as $L^{p}(\R^+, L^2(\sd))$ i.e,
 as the $L^p$ space on $\R^+$ taken with respect to the
 measure $d\mu _{\frac{d}{2}-1}$ of functions taking values in the
 Hilbert space $\mathcal{H}= L^2(\sd )$. Thus given $f \in L^{p , 2}(\R^d)$
 and $s \in \R^+$ we consider $f(s)$ defined on $\sd$ by $ f(s)(\om)=
 f(s\om) $ as an element of $\mathcal{H}$.
 Note that $$\| f \| _{L^{p,2}(\R^d)}= \Big ( \int_0 ^{\infty} \|f(s)\|^p_{\mathcal{H}} d\mu_{\frac{d}{2}-1}(s)\Big
 )^{\frac{1}{p}}.$$
 For $r,s \in \R^+$, $r\neq s$ consider the linear operator $R_j(r,s): L^2(\sd)\rightarrow
 L^2(\sd)$ defined by
 $$
 R_j(r,s)\varphi (\om) = \int _{\sd}R_j(r\om , s\om ')\varphi(\om ')d \om '.
 $$
 It is then clear that $R_j(r,s)$ is a bounded operator on $L^p(\sd)$ and
 $$\|R_j(r,s)\|_{Op} \leq \sup_{\om} \int_{\sd} |R_j(r\om, s\om ')|d\om '.$$
 We can now view of $R_j$ as an integral operator defined on $L^p(\R^+, \mathcal{H})$. Indeed,
 we see that
 \begin{eqnarray}
 R_jf(r\om)=\int_0 ^\infty(R_j(r,s)f(s))(\om)s^{d-1}ds
 \end{eqnarray}
 and hence we will study $R_j$ by means of singular integral
 operators on $L^p(\R^+, \mathcal{H})$. \\

  For any $ \alpha \geq -\frac{1}{2}$ we consider
  $\R^+$ equipped with the measure $ d\mu _{\alpha}(r) = r^{2\alpha + 1}dr $ as a homogeneous space.
  Let $B(a,b)$ stand for the ball of radius $b>0$  centered at $a\in
  \R^+$. We say that a non negative and locally integrable
  function $w$ is in $ A^\alpha _p(\R ^+)$, $1<p<\infty $ if $w$
  satisfies
   $$\Big ( \frac{1}{\mu _\alpha (Q)} \int_{Q} w(r) d\mu_\alpha \Big )\Big ( \frac{1}{\mu _\alpha (Q)} \int_{Q} w(r)^{-\frac{p'}{p}} d\mu_\alpha \Big )^{p-1}
  \leq C$$
  for all intervals $Q\subset \R^+ $.
  We consider a singular integral operator $T$
  acting on vector valued ($\mathcal{H}$-valued ) functions on
  $\R^+$ associated to an operator valued kernel $K$
  which is defined on $ \R^+ \times \R^+ $ and taking values in
  $\mathcal{L}(\mathcal{H}, \mathcal{H})$, the space of bounded linear
  operators on $ \mathcal{H }$. That is to say
  $$Tf(r)= \int_0^\infty K(r,s)f(s)d\mu_{\alpha}(s) $$
  for all bounded and compactly supported vector valued functions
  $ f $ and $ r \notin supp f$. The following theorem tells us the
  boundedness of $ T$ under some conditions on $ K $.

  \begin{thm}
  Suppose that $T $ defined as above is bounded
  on $L^2(\R^+, \mathcal{H})$. Assume that $ K $
  satisfies the following conditions
  \begin{enumerate}
  \item $\|K(r,s)\| _{op} \leq C_1 \Big ( \mu_\alpha(B(r, |r-s|))\Big )^{-1} $ and

  \item $\| \pa _r K(r,s)\|_{op} \leq C_2|r-s|^{-1} \Big (\mu_\alpha(B(r,|r-s|))\Big )^{-1}$ for $r \neq s$.
  \end{enumerate}
  Then $T$ is bounded on $L^p(\R^+, \mathcal{H} )$, $1<p<\infty$ and weak type
  $(1,1)$. More generally, if $w \in A^\alpha _p(\R
  ^+)$, then we have
  $$ \|Tf\|_{L^p(\R^+, \mathcal{H}, wd\mu _\alpha)} \leq C \|f\| _{L^p(\R^+, \mathcal{H}, wd\mu
  _\alpha)}$$ for $1<p<\infty$.
 \end{thm}
 \begin{proof} By an easy calculation using the hypothesis on $K$, we
 can show that there exists a positive constant $C$ such that for
 all $s, t \in \R^+ $
 $$\int_{|r-s|> 2|s-t|}\| K(r,s)-K(r,t)\|_{op}d\mu_{\alpha}(r) \leq C $$
 and
 $$\int_{|r-s|> 2|s-t|}\| K(s,r)-K(t,r)\|_{op}d\mu_{\alpha}(r) \leq C. $$
 For the un-weighted case, the boundedness of
 $T$ is proved in Theorem 1.1 in \cite{GLY} and for the
 weighted case, the boundedness of $T$ can
 be proved by imitating the proof of Theorem 7.11, page no. 144 in \cite{DJ} by
 changing the notations and notions(definitions) into our setting.
 \end{proof}
 As we have observed, the Riesz transforms can be considered as
 operators defined on $L^2(\R^+, \Hc; d\mu _{\frac{d}{2}-1})$,
 $\Hc =L^2(\sd)$:
 $$
 R_jf(r) = \int_0 ^\infty R_j(r,s)f(s)d\mu_{\frac{d}{2}-1}(s).
 $$
 Then $R_j$ is clearly bounded on $L^2(\R^+, \Hc; d\mu
 _{\frac{d}{2}-1})$. So, in order to prove Theorem 1.1, all we need
 is the following proposition giving estimates on the operator
 valued kernels $R_j(r,s)$.
 \begin{prop}
For any $r \neq s$ we have\\

(1) $ \|R_j(r,s) \|_{Op} \leq C_1 \Big (\mu_{\frac{d}{2}-1}(B(r,|r-s|))\Big )^{-1} $\\

(2) $\| \pa_r R_j(r,s)\|_{Op} \leq C_2 |r-s|^{-1}\Big (\mu
 _{\frac{d}{2}-1} (B(r,|r-s|))\Big )^{-1}$
\end{prop}
\begin{proof} As we have already observed
$$\|R_j(r,s)\|_{Op} \leq \sup_{\om} \int_{\sd} |R_j(r\om, s\om ')|d\om '.$$
Using the estimate (1) of Proposition 2.1, we get
\begin{eqnarray}
\|R_j(r,s)\|_{Op} \leq C \sup_{\om} \int_{\sd} |r\om - s\om
'|^{-d} d\om '
\end{eqnarray}
Integrating in polar coordinates, the last integral is a constant
multiple of
$$ \int_0 ^{\pi} (r^2+s^2-2rs \cos\theta )^{-\frac{d}{2}} (\sin\theta)^{d-2}d\theta $$
which is bounded by
\begin{eqnarray*}
c \int_0 ^{1} (r^2+s^2-2rsu )^{-\frac{d}{2}}
(1-u)^{\frac{d-3}{2}}du.
\end{eqnarray*}
In order to estimate this we make use of the following lemma (see
Lemma 5.3 in \cite{CR1})
\begin{lem}
Let $c\geq \frac{1}{2}$, $0<B<A$ and $\lambda >0$. Then
$$\int_0 ^1 \frac{(1-u)^{c-\frac{1}{2}}}{(A-Bu)^{c + \lambda + \frac{1}{2}}} du \leq \frac{C}{A^{c+\frac{1}{2}}(A-B)^{\lambda}}.$$
\end{lem}
Appealing to this lemma with $A= r^2 +s^2$, $B=2rs$,
$c=\frac{d}{2}-1$ and $\lambda = \frac{1}{2}$ we obtain
$$ \| R_j(r,s) \|_{Op} \leq C |r-s|^{-1} (r^2 + s^2)^{-(\frac{d-1}{2})}.$$
The desired estimate follows as $|r-s|(r^2+s^2)^{\frac{d-1}{2}}$
is comparable to $\mu_{\frac{d}{2}-1} (B(r,|r-s|))$. In order to
get the estimate on the derivative we note that
\begin{eqnarray*}
  \frac{\pa}{\pa r}R_j(r,s)\varphi(\om) & = & \frac{\pa}{\pa r}\int _{\sd}R_j(r\om,s\om ')\varphi(\om')d\om ' \\
                                        &= & \sum_{i=1} ^d \int _{\sd} \frac{\pa}{\pa x_i}R_j(r\om,s\om ')\varphi(\om') \om _i d\om '
 \end{eqnarray*}
 The estimate (2) of Proposition 2.1 can be used to bound the
 operator norm of $\frac{\pa}{\pa r}R_j(r,s)$. This leads to
 $$\|\frac{\pa}{\pa r}R_j(r,s) \|_{Op}\leq C \sup_{\om} \int_{\sd} |r\om - s\om '|^{-d-1} d\om '.$$
 As before, using Lemma 2.4 we get the desired estimate.
\end{proof}
 This completes the proof of Theorem 1.1.
 \subsection{Another proof of Theorem 1.1:}
 In this subsection we give another proof of Theorem 1.1 following
 an idea of Rubio de Francia. This method described briefly in
 \cite{RF} is based on an extension of a theorem of Marcinkiewicsz and
 Zygmund as expounded in Herz and Riviere \cite{HR}. Indeed, we make use
 of the following Lemma which can be found in \cite{HR}

 \begin{lem}
 Let $( G, \mu )$ and $( H, \nu )$ be arbitrary measure spaces and $T:
 L^p(G)\rightarrow L^p(G)$ a bounded linear operator. Then if $p\leq q\leq 2 $
 or $p\geq q \geq 2$, there exists a bounded linear operator $\tilde{T}: L^p(G; L^q(H))\rightarrow L^p(G;L^q(H))
 $ with $\|  \tilde{T}  \| \leq\|  T  \| $ such that for $g\in L^p(G; L^q(H)) $ of the form $g(x,\xi)=f(\xi)u(x)$
 where $f\in  L^p(G)$ and $u\in L^q(H) $ we have
 $$(\tilde{T}g)(\xi, x)=(Tf)(\xi)u(x) .$$
 \end{lem}
 The idea of Rubio de Francia is as follows (we are indebted to
 Gustavo Garrigos for bringing this to our attention). Suppose $
 T: L^p(\R^d, dx) \rightarrow L^p(\R^d, dx) $ is a bounded
 linear operator. Then by the lemma of Herz and Riviere, it has an extension $
 \tilde{T} $ to $ \mathcal{H} $ valued functions on $\R^d$ where
 $ \mathcal{H} $ is the Hilbert space $ L^2(K)$, $K=SO(d)$.
 Moreover, the extension satisfies $(\tilde{T}\tilde{f})(x,k) =
 Tg(x)h(k)$ if $\tilde{f}(x,k)=g(x)h(k) $, $x \in \R^d $, $k\in SO(d) $. Given $f \in L^p(\R^d, dx) $ consider $\tilde{f}(x, k) = f(k.x)
 $. Then $\int _{\R^d}(\int _K | \tilde{f}(x,k)|^2 dk )^{\frac{p}{2}}dx $ can be calculated as follows.
 If $x=r\om $ , $\om \in \sd $, $\tilde{f}(x,k)=f(rk \cdot\om) $ and
 hence
 \begin{eqnarray}
  \int _K |\tilde{f}(x,k)|^2dk =
  \int_{K_{\om}}\Big (\int_{K/K_{\om}}|f(r k\cdot \om )|^2 d \mu \Big )d\nu
 \end{eqnarray}
 where $K_{\om}=\{ k\in K: k\cdot \om = \om \} $ is the isotropy subgroup of $ K $, $ d\nu $ is the Haar
 measure on $K_{\om} $ and $d\mu $ is the $K_{\om} $ invariant measure on $K/K_{\om} $ which
 can be identified with $\sd $. Hence
 \begin{eqnarray}
   \int _K |\tilde{f}(x,k)|^2 dk = c\int _{\sd} |f(r\om)|^2d\sigma
   (\om).
 \end{eqnarray}
 Therefore,
  \begin{eqnarray}
  & & \int_{\R^d} \Big ( \int _K |\tilde{f}(x,k)|^2 dk\Big )^{\frac{p}{2}}dx
   =  c' \int _0 ^{\infty} \Big ( \int _{\sd} |f(r\om)|^2d\sigma
   (\om)\Big )^{\frac{p}{2}}r^{d-1}dr.
 \end{eqnarray}

 Let us define $\rho (k)f(x)=f(k\cdot x) $ so that $\tilde{f}(x,k)=\rho (k) f(x) $. If $ T $ commutes with rotation
 i.e. $T\rho (k)= \rho (k)T $ then
 $$\tilde{T} \tilde{f}(x,k) = T(\rho (k)f)(x) = \rho (k) (Tf)(x) = (Tf)(k\cdot x).$$
 The boundedness of $ \tilde{T} $ on $ L^p(\R^d, \mathcal{H}) $ gives
 \begin{eqnarray}
  & & \int_{\R^d} \Big ( \int _K | T f(k\cdot x)|^2 dk \Big )^{\frac{p}{2}}dx  \leq  C \int_{\R^d} \Big ( \int _K | f(k\cdot
 x)|^2 dk \Big )^{\frac{p}{2}}dx
 \end{eqnarray}
 which translates into the mixed norm estimate for $ T $.\\

 Given a unit vector $u \in \sd $ let us consider the operator $T_uf = \sum_{j=1}^d u_jR_jf(x)$ where
 $R_j = A_jH^{- \frac{1}{2}} $ are the Hermite Riesz transforms. This operator $T_u $ is not
 rotation invariant but has a nice transformation property under the
 action of $SO(d) $. Indeed,
 $$T_uf(x)= (x\cdot u + u \cdot \N )H^{-\frac{1}{2}}f(x)$$
 and as $H^{-\frac{1}{2}} $ commutes with $\rho (k) $ it follows that
 $$T_u\rho (k) f = \rho (k) T_{k\cdot u}f $$
 or
 $$T_{k^{-1}\cdot u}\rho (k) f = \rho (k) T_{u}f .$$
 This leads us to
 $$ T_uf(k\cdot x)= \sum_{j=1}^d (k^{-1}\cdot u)_j R_j(\rho(k) f)(x).$$
 We make use of this in proving Theorem 1.1.\\

 The operator $ R_j$ are singular integral operators and hence
 bounded on $L^p(\R^d, wdx) $ for any weight function $w \in A_p(\R^d) $, $1<p<\infty $. By the lemma of
 Herz and Riviere, $R_j $ extends as a bounded operator on $L^p(\R^d, \mathcal{H}; wdx ) $ where $\mathcal{H}=L^2(\sd) $. When
 $ w $ is radial, it can be easily checked that
 \begin{eqnarray}
 & & \int_{\R^d}\Big ( \int _K |\rho(k)f(x)|^2dk\Big )^{\frac{p}{2}}w(x)dx \notag \\
 & & = c\int_0 ^{\infty} \Big ( \int_{\sd}|f(r\om)|^2 d\sigma (\om)\Big
 )^{\frac{p}{2}}w(r)r^{d-1}dr
 \end{eqnarray}
 Moreover, by the result of Duoandikoetxea (Theorem 3.2 in \cite{DMOS}), any radial $w \in A_p(\R^d) $ if and only if $w(r) \in A_p ^{\frac{d}{2}-1}(\R^+)  $.
 From the identity
 $$T_uf(k\cdot x)= \sum _{j=1} ^{d} (k^{-1}\cdot u)_jR_j(\rho (k)f)(x) $$
 we obtain
 \begin{eqnarray*}
    &  \Big ( \int_{\R^d} \Big (\int _K |T_uf(k\cdot x)|^2dk
   \Big)^{\frac{p}{2}}w(x)dx \Big )^{\frac{1}{p}}& \\
    & \leq  c \sum _{j=1} ^{d}\Big ( \int_{\R^d} \Big (\int _K |\rho (k)f(x)|^2dk
   \Big)^{\frac{p}{2}}w(x)dx \Big )^{\frac{1}{p}}&
 \end{eqnarray*}
 which translates into the required inequality of Theorem 1.1 by
 taking $ u $ to be coordinate vectors.

 \subsection{Laguerre Riesz transforms}
 For each $\alpha\geq - \frac{1}{2}$ we consider the Laguerre
 differential operator
 $$ L_\alpha = - \frac{d ^2}{dr^2} + r^2 - \frac{2\alpha +1}{r} \frac{d}{dr}$$
 whose normalized eigenfunctions are given by
 $$\psi_k ^\alpha (r) = \Big ( \frac{2\Gamma(k+1)}{\Gamma (k+\alpha + 1)} \Big )^{\frac{1}{2}} L _k ^\alpha (r^2)e^{-\frac{1}{2}r^2}$$
 where $ L _k ^\alpha (r)$ are Laguerre polynomials of type $\alpha$. These functions
 form an orthonormal basis for $L^2(\R^+, d\mu _\alpha ) $. The operator $L_\alpha $ generates the
 semigroup $T_t ^\alpha = e^{-tL_\alpha} $ whose kernel is given by
 \begin{eqnarray}
 K_t ^\alpha (r,s) = \sum_{k=0}^\infty e^{-(4k+2\alpha+2)t}\psi_k
 ^\alpha (r)\psi_k ^\alpha (s).
 \end{eqnarray}
 The generating function identity (1.1.47 in \cite{ST}) for Laguerre
 functions gives the explicit expression
 \begin{eqnarray}
 K_t ^\alpha (r,s)= (\sinh 2t)^{-1} e^{- \frac{1}{2} (\coth 2t)(r^2+s^2)}(rs)^{- \alpha} \mathcal{I}_\alpha \Big (\frac{rs}{\sinh 2t} \Big )
 \end{eqnarray}
 where $\mathcal{I}_\alpha = e^{-i\frac{\pi}{2}\alpha}J_\alpha (iz)$ is the modified Bessel function.\\

 The Laguerre Riesz transforms $R^\alpha = (\frac{\pa}{\pa r} + r )L _\alpha ^{-\frac{1}{2}} $ have been studied in  \cite{NS} and it is
 known that they are bounded on $ L^p(\R^+, d\mu _\alpha )$, $1<p<\infty $. Here we are interested
 in a vector valued inequality for the sequence of Riesz
 transforms $ R^{\alpha + m}$ where $\alpha = \frac{d}{2}-1 $.
 \begin{thm}
 Let $ d \geq 2$ and $\alpha = \frac{d}{2}-1$. Then for any $1<p<\infty $, $w \in A_p ^{\alpha}(\R ^+)$ we have
 $$\int_0 ^\infty \Big ( \sum_{m=0} ^\infty r^{2m} | R^{\alpha + m} \widetilde{f}_m(r) |^2\Big
 )^{\frac{p}{2}}w(r)d\mu_{\frac{d}{2}-1}(r)$$ $$ \leq C\int_0 ^\infty \Big (\sum_{m=0} ^\infty  |f_m(r)|^2 \Big
 )^{\frac{p}{2}}w(r)d\mu_{\frac{d}{2}-1}(r) $$
 where $ f_m \in  L^p(\R^+, d\mu _\alpha )$, $\widetilde{f}_m(r)= r^{-m}f_m(r)$.
 \end{thm}
 Before proving this theorem we remark that in \cite{CR} where the
 authors have studied the boudedness of Hermite Riesz transforms in
 terms of polar coordinates, the above theorem has been proved
 first from which our main theorem can be deduced. Here we have
 already proved Theorem 1.1 and now we will show how the above
 result can be deduced.\\

 Coming to the proof of the above theorem, consider the vector of
 the Riesz transforms
 $$Rf(x)=(R_1f(x), R_2f(x),\cdots ,R_df(x)).$$
 If $\nabla $ stand for the gradient then it follows that $Rf(x)=(x+\nabla )H^{-\frac{1}{2}}f(x)$. Let $\Hc _m$
 stand for the space of spherical harmonics of degree $m$ of
 dimension $d(m)$. Fix an orthonormal basis $Y_{m,j}, j= 1,2,\ldots,d(m)$ for $\Hc _m$ consisting of
 real valued spherical harmonics. Any $f\in L^{p,2}(\R^d)$ has an expansion
 $$f(r\om)= \sum_{m=0} ^{\infty} \sum_{j=1} ^{d(m)}f_{m,j}(r)Y_{m,j}(\om).$$
 In view of the Hecke - Bochner formula (Theorem 3.4.1 in \cite{ST}) for
 the Hermite projections, the operator $H^{-\frac{1}{2}}$ preserves each $\Hc _m$ and
 consequently
 $$H^{-\frac{1}{2}}f(r\om) =\sum_{m=0} ^{\infty} \sum_{j=1} ^{d(m)}F_{m,j}(r)Y_{m,j}(\om)  .$$
 Here $F_{m,j}(r)$ are given by
 $$ F_{m,j}(r)= \int_{\sd} H^{-\frac{1}{2}}f(r\om)Y_{m,j}(\om) d\om.$$
 The gradient $\nabla $ can be split into radial and angular parts as
 follows: $\nabla = \frac{1}{r}\nabla _0+\om \frac{\pa}{\pa r}$ if $x = r\om $, $\om \in \sd $ is the polar decomposition. Here $\nabla _0 $
 acts in the $\om$ - variable, see \cite{DX}. In view this formula we see
 that
 $$(x+\nabla )(F_{m,j}(r)Y_{m,j}(\om)) = \Big (r+\frac{\pa}{\pa r}\Big )F_{m,j}(r)Y_{m,j}(\om)\om + \frac{1}{r}F_{m,j}(r)\nabla _0 Y_{m,j}(\om).$$
 Thus we see that
 $$Rf(r\om)= \Big ( \sum_{m=0} ^{\infty} \sum_{j=1} ^{d(m)}\Big (r+\frac{\pa}{\pa r}\Big )F_{m,j}(r)Y_{m,j}(\om) \Big )\om
 + \frac{1}{r}\sum_{m=0} ^{\infty} \sum_{j=1} ^{d(m)}F_{m,j}(r)U_{m,j}(\om) $$
 where $U_{m,j}(\om) = (U^1 _{m,j}(\om),\cdots, U^d _{m,j}(\om))= \nabla _0 Y_{m,j}(\om) $.
 We make use of the following facts: (i) $ \om .\nabla _0 Y_{m,j} = 0 $ and
 (ii) $ \int _{\sd} \nabla _0 Y_{m,j}\cdot \nabla _0 Y_{m',j'}d\om= m(m+d-2) \delta_{m,m'}\delta_{j,j'} $.
 For a proof of these facts see Lemma 2.2 in \cite{PX}. Therefore, $ |Rf(r\om) |^2 = \sum _{j=1}^d |R_jf(r\om)|^2$ is given by the sum of
 $$ \Big | \sum_{m=0} ^{\infty} \sum_{j=1} ^{d(m)}\Big (r+\frac{\pa}{\pa r}\Big )F_{m,j}(r)Y_{m,j}(\om)\Big |^2$$
 and
 $$\frac{1}{r^2}\sum_{k=0}^d \Big |\sum_{m=0} ^{\infty} \sum_{j=1} ^{d(m)}F_{m,j}(r)U^k _{m,j}(\om)\Big |^2 $$
 where we have used the facts (i). Integrating over $\sd $ and using
 the fact (ii) we see that
 $$\int _{\sd}\Big ( \sum_{j=1}^d |R_jf(r\om)|^2\Big ) d\om$$
 is given by the sum of
 $$  \sum_{m=0} ^{\infty} \sum_{j=1} ^{d(m)}\Big | \Big (r+\frac{\pa}{\pa r}\Big )F_{m,j}(r)\Big |^2$$
 and
 $$ \sum_{m=0} ^{\infty} \sum_{j=1} ^{d(m)}\frac{1}{r^2}m(m+d-2) |F_{m,j}(r)|^2.$$
 Therefore, our main theorem gives the following two inequalities :
 \begin{align}
 \int _0 ^\infty \Big ( \sum_{m=0} ^{\infty} \sum_{j=1} ^{d(m)}\Big | \Big (r+\frac{\pa}{\pa r}\Big )F_{m,j}(r)\Big |^2\Big
 )^{\frac{p}{2}}w(r)r^{d-1}dr & & \notag \label{A}\\
 \leq  C \int _0 ^\infty \Big ( \sum_{m=0} ^{\infty} \sum_{j=1} ^{d(m)} | f_{m,j}(r)|^2\Big
 )^{\frac{p}{2}}w(r)r^{d-1}dr
  \end{align}
 and
 \begin{align}
 \int _0 ^\infty \Big (\sum_{m=0} ^{\infty} \sum_{j=1} ^{d(m)}\frac{1}{r^2}m(m+d-2) |F_{m,j}(r)|^2\Big
  )^{\frac{p}{2}}w(r)r^{d-1}dr& & \notag \label{B}\\
   \leq C \int _0 ^\infty \Big ( \sum_{m=0} ^{\infty} \sum_{j=1} ^{d(m)} | f_{m,j}(r)|^2\Big )^{\frac{p}{2}}w(r)r^{d-1}dr.
 \end{align}
 We will now restate the first of the above inequality in terms of
 Laguerre-Riesz transforms which will immediately prove Theorem
 2.6.\\

 Appealing to Hecke - Bochner formula for the Hermite projections
 it is easy to see that
 \begin{eqnarray}
 e^{-tH}(f_{m.j}Y_{m,j})(r\om)= c_d r^m Y_{m,j}(\om)T_t ^{\alpha+m}\widetilde{f}_{m,j}(r)
 \end{eqnarray}
 where $\widetilde{f}_{m,j}(r)= r^{-m}f_{m,j}(r)$, $\alpha= \frac{d}{2}-1 $ and $T_t ^{\alpha+m}= e^{-tL_{\alpha + m}} $. For a different proof of this see
 Proposition 3.2 in \cite{LT}. Consequently,
\begin{eqnarray*}
  F_{m,j}(r) &=& c_d r^m \frac{1}{\sqrt{\pi}}\int_0 ^{\infty} T_t ^{\alpha+m}\widetilde{f}_{m,j}(r)t^{-\frac{1}{2}}dt \\
   &=& c_d r^m L_{\alpha + m} ^{-\frac{1}{2}}\widetilde{f}_{m,j}(r).
\end{eqnarray*}
This leads to the conclusion that
 $$\Big ( r + \frac{\pa}{\pa r} \Big )F_{m,j}(r) = c_d r^m R^{\alpha + m} \widetilde{f}_{m,j}(r)+c_d mr^{m-1}L_{\alpha + m} ^{-\frac{1}{2}}\widetilde{f}_{m,j}(r).$$
 As $mr^{m-1}L_{\alpha + m} ^{-\frac{1}{2}}\widetilde{f}_{m,j}(r) =
 \frac{m}{r}F_{m,j}(r)$, from the main theorem we get the inequality
 $$\int_0 ^{\infty}\Big (\sum_{m=0} ^{\infty} \sum_{j=1} ^{d(m)} r^{2m}|R^{\alpha + m}\widetilde{f}_{m,j}(r)|^2 \Big )^{\frac{p}{2}}w(r)r^{d-1}dr$$
 $$ \leq C \int _0 ^\infty \Big ( \sum_{m=0} ^{\infty} \sum_{j=1} ^{d(m)} | f_{m,j}(r)|^2\Big )^{\frac{p}{2}}w(r)r^{d-1}dr.$$
 This completes the proof of Theorem 2.6.
 \begin{rem}
 We also have the inequality
\begin{align}
  \int_0 ^{\infty}\Big (\sum_{m=0} ^{\infty} \sum_{j=1} ^{d(m)} m^2 r^{2m-2}|L_{\alpha + m}^{-\frac{1}{2}}\widetilde{f}_{m,j}(r)|^2
  \Big )^{\frac{p}{2}}w(r)r^{d-1}dr& & \notag{}\\
  \leq  C \int _0 ^\infty \Big ( \sum_{m=0} ^{\infty} \sum_{j=1} ^{d(m)} | f_{m,j}(r)|^2\Big )^{\frac{p}{2}}w(r)r^{d-1}dr.
 \end{align}
 \end{rem}
 \noindent
 which has been proved in \cite{CR} by first establishing uniform estimates
 for the kernels of the operators taking $f_{m,j}$ into $\frac{m}{r}L_{\alpha + m}^{-\frac{1}{2}}\widetilde{f}_{m,j}(r)$.
 \subsection{Another proof of the inequality \eqref{A}}
 In this subsection we give a different proof of the inequality
 \begin{align*}
 \int _0 ^\infty \Big ( \sum_{m=0} ^{\infty} \sum_{j=1} ^{d(m)}\Big | \Big (\frac{\pa}{\pa r}+r \Big )F_{m,j}(r)\Big |^2\Big
)^{\frac{p}{2}}w(r)r^{d-1}dr& & \\
 \leq   C \int _0 ^\infty \Big ( \sum_{m=0} ^{\infty} \sum_{j=1} ^{d(m)} | f_{m,j}(r)|^2\Big )^{\frac{p}{2}}w(r)r^{d-1}dr.
 \end{align*}
 Let us set $T_t = e^{-tH} $ which is an integral operator with kernel $K_t(x, y) $.
 Note that $K_t $ depends only on $|x|$, $|y|$ and $x \cdot y$. We write
 $$K_t(r,s,u)= c_d (\sinh 2t)^{-\frac{1}{2}} e^{-\frac{1}{4} (\coth 2t)(r^2+s^2)+(\frac{1}{\sinh 2t})rsu} $$
 so that $ K_t (x,y)= K_t(r,s,u)$ with $r=|x|$, $s=|y|$ and $u= x'\cdot y' $. With this notation
 $$F_{m,j}(r) = \int_{\sd} \int_0 ^{\infty } T_tf(r\om) Y_{m,j}(\om) t^{-\frac{1}{2}}dtd \om .$$
 Let $ \mathcal{P}_m ^{\frac{d}{2}-1}(u)$ stand for the ultraspherical polynomial normalised so
 that $ \mathcal{P}_m ^{\frac{d}{2}-1}(1) = 1$. Recall that Funk - Hecke  formula gives (See \cite{DUX})
 $$\int_{\sd } \varphi (x'\cdot y') Y_{m,j}(y')dy' = Y_{m,j}(x')\int_{-1} ^{1} \varphi(u) \mathcal{P}_m ^{\frac{d}{2}-1}(u) (1-u^2)^{\frac{d-3}{2}}du.$$
 Applying this formula we see that
 $$\int_{\sd} T_tf(rx')Y_{m,j}(x') dx' $$ $$ = \int_0 ^{\infty } \int_{\sd} f(sy') \Big ( \int_{\sd} K_t(rx', sy')Y_{m,j}(x') dx' \Big ) s^{d-1}dy'ds$$
 $$   = \int _0 ^{\infty} \Big ( \int_{-1} ^{1} K_t(r,s,u)\mathcal{P}_m ^{\frac{d}{2}-1}(u)(1-u^2)^{\frac{d-3}{2}}du \Big ) f_{m,j}(s)s^{d-1}ds. $$
 On the other hand we also have
 $$ \int_{\sd} K_t(x,y)\mathcal{P}_m ^{\frac{d}{2}-1}(x'\cdot y')dy'
  = \int_{-1} ^{1} K_t(r,s,u)\mathcal{P}_m ^{\frac{d}{2}-1}(u)(1-u^2)^{\frac{d-3}{2}}du. $$
 Consequently,
 \begin{eqnarray*}
 & &\int_{\sd} T_tf(rx')Y_{m,j}(x') dx'\\ & & = \int _0 ^{\infty}\Big ( \int_{\sd} K_t(x,y)\mathcal{P}_m ^{\frac{d}{2}-1}(x'\cdot y') f_{m,j}(s)\Big ) dy' s^{d-1}ds.
 \end{eqnarray*}
 If we define $ K_{t,m} (x,y) = K_t(x,y)\mathcal{P}_m ^{\frac{d}{2}-1}(x'\cdot y') $ then we have
 \begin{align}
 F_{m,j}(r) = \int _0 ^{\infty} K_m (r,s)f_{m,j}(s)s^{d-1}ds
 \end{align}
 where
 \begin{align}
 K_m(r,s) = \int _0 ^{\infty} \Big (\int_{\sd} K_{t,m} (x,y) dy'
 \Big ) t^{-\frac{1}{2}}dt.
 \end{align}
 We will use this expression to estimate the kernel $K_m $ and their
 derivatives.
 \begin{prop}
 \item{(1)} $\Big |\Big ( \frac{\pa}{\pa r} + r \Big ) K_m (r,s) \Big | \leq
 C \Big ( \mu_{\frac{d}{2}-1} ( B(r, |r-s| ) )\Big )^{-1} $

 \item{(2)} $\Big | \frac{\pa}{\pa r} \Big ( \frac{\pa}{\pa r} + r \Big
 ) K_m (r,s) \Big | \leq C |r-s|^{-1} \Big ( \mu_{\frac{d}{2}-1} ( B(r,
 |r-s| ) )\Big )^{-1} $
 \end{prop}
 \begin{proof}
 $ \Big ( \frac{\pa}{\pa r} + r \Big ) K_m (r,s) = \int _0
 ^{\infty} \int_{\sd}\Big ( \frac{\pa}{\pa r} + r \Big ) K_t(x,y)\mathcal{P}_m
 ^{\frac{d}{2}-1}(x'\cdot y') dy'  t^{-\frac{1}{2}}dt .$ \\
 Since $\frac{\pa}{\pa r} K_t(x,y) = \sum _{j=1} ^{d} x_j \frac{\pa}{\pa x_j}K_t(x,y) $ and $ r = \sum _{j=1} ^{d} x'_jx_j $
 we need to estimate
 $$\sum _{j=1} ^{d} x'_j \int _0 ^{\infty} \int_{\sd}\Big ( \frac{\pa}{\pa x_j} + x_j \Big ) K_t(x,y)\mathcal{P}_m
 ^{\frac{d}{2}-1}(x'\cdot y') dy' t^{-\frac{1}{2}}dt $$
 which is given by
 $$ \sum _{j=1} ^{d} x'_j \int_{\sd}R_j (x,y)\mathcal{P}_m ^{\frac{d}{2}-1}(x'\cdot y') dy' $$
 where $R_j(x,y)$ are the kernels of the Riesz transforms. The estimates
 on $R_j(x,y) $ can be used along with the fact that $| \mathcal{P}_m ^{\frac{d}{2}-1}(u)| \leq 1 $ to get the
 estimate $$\int _{\sd} |rx'-sy'|^{-d} dy'$$
 which gives the required estimate. The estimate on the gradient
 is proved in a similar way using the estimate for $\frac{\pa}{\pa x_i} R_j(x,y) $.
 \end{proof}
 \begin{rem}
 It is also possible to prove uniform estimates for the
 kernels of the operators taking $f_{m,j} $ into $\frac{m}{r} F_{m,j} $. These estimates lead
 to a different proof of Theorem 1.1 as in \cite{CR}.
 \end{rem}


\section{Riesz transforms for the special Hermite operator}
\subsection{Special Hermite operator and Riesz transforms}
In this section we recall preliminaries about special Hermite
expansion (which are discussed in Chapter 1 and Chapter 2. of
\cite{ST}) required for the proof of Theorem 1.2. For $f \in
L^2(\C^d)$, the special Hermite expansion of $ f $ is given by
$$ f(z) = (2\pi)^{-d} \sum_{k=0} ^\infty f \times \varphi_k (z) $$
where $\varphi_k =
L_k^{d-1}(\frac{1}{2}|z|^2)e^{-\frac{1}{4}|z|^2} $ are the
Laguerre functions, $L_k^{d-1}$ being Laguerre polynomials of type
$(d-1) $. And the convolution in the above is called the twisted
convolution defined as follows: for given $f, g \in L^2(\C^d), $
$$f \times g(z) = \int_{\C^d} f(z-w)g(w)e^{\frac{i}{2}\Im (z\cdot
\overline{w})}dw.$$ Now consider the the special Hermite operator
$L$ given by
$$L= -2\sum_{j=1}^d (Z_j\overline{Z}_j + \overline{Z}_jZ_j) = -\Delta_z+\frac{1}{4}|z|^2- i N $$
where $\Delta_z$ is the Laplacian on $\C^d$, $ Z_j = \Big
(\frac{\pa}{\pa z_j}+\frac{1}{4}\overline{z}_j \Big )$, $Z_j =
\Big (\frac{\pa}{\pa \overline{z}_j}-\frac{1}{4}z_j \Big ) $ and
$N=\sum_{j=1}^d \Big ( x_j\frac{\pa}{\pa y_j} - y_j \frac{\pa}{\pa
x_j} \Big )$ with $z=x+iy \in \C^d $. It is known that $ L $ is a
positive, symmetric and elliptic operator on $\C^d$(for more
details see Chapter 1 and Chapter 2. in \cite{ST}). Special
Hermite expansion is the spectral decomposition of $L$ and $
(2\pi)^{-d} (f \times  \varphi_k) $ is the projection of $f$ onto
the eigenspace corresponding to the eignevalue $(2k+d)$. If
$e^{-tL} $ is the special Hermite semigroup generated by $L$ then
 \begin{align}
 e^{-tL}f(z) = (2\pi)^{-d}\sum_{k=0}^\infty e^{-(2k+d)t}f \times
 \varphi_k(z)
 \end{align}
 for functions $f \in L^2(\C^d) $.
 For Schwartz class functions $f$, it is clear that $e^{-tL}f(z)
 =f\times p_t(z) $ where
 $$p_t(z)=(2\pi)^{-d} \sum_{k=0}^\infty e^{-(2k+d)t}\varphi_k(z). $$
 The generating function for Laguerre functions of type $(d-1)$
 leads to the explicit formula
 \begin{align}
 p_t(z)= (2\pi)^{-d}(\sinh t)^{-d}e^{-\frac{1}{4}|z|^2 \coth t}.
 \end{align}
 The operator $ L^{-\frac{1}{2}}$ can be defined using spectral
 theorem, which is also given by
 $$ L^{-\frac{1}{2}}f(z)= \frac{1}{\sqrt{\pi}}\int_0 ^\infty e^{-tL}f(z)t^{-\frac{1}{2}}dt. $$
 Consequently the Riesz transforms $S_j$ for the special Hermite
 operator are defined by
 \begin{eqnarray}
   S_j f(z)  =  Z_j L^{-\frac{1}{2}}f(z)
            &= & \frac{1}{\sqrt{\pi}}\int_0 ^\infty Z_j
               e^{-tL}f(z)t^{-\frac{1}{2}}dt  \\
           & = & f \times s_j(z)\notag{}
 \end{eqnarray}
 where $$ s_j(z)=\frac{1}{\sqrt{\pi}}\int_0 ^\infty Z_j p_t (z)
 t^{-\frac{1}{2}}dt .$$ It is known that $S_j$ are bounded on $
 L^p(\C^d)$,\; $1<p<\infty $ and weak type $(1,1)$ (see Theorem 2.2.2 in \cite{ST}). The mixed norm estimates
 for the Riesz transforms $S_j$ will be proved in subsection 3.3.
 \subsection{Bigraded spherical harmonics}
 In order to study mixed norm estimates for the Riesz transforms
 associated to special Hermite expansions, we need to make use of
 several properties of bigraded spherical harmonics. If
 $\mathcal{H}_{N}$ stands for the space of (ordinary) spherical
 harmonics of degree $N$ on $\std$ then we have an action of the
 unitary group $U(d)$ on $\mathcal{H}_{N}$. Under this action $\Hc
 _N$ decomposes into irreducible pieces $\mathcal{H}_{m,n}$ where
 $m$ and $n$ are non-negative integers with $m+n = N$. The
 members of $\Hc _{m,n}$ are called bigraded spherical harmonics.\\

 Let $\D _z = \D _x +\D _y $ stand for the Laplacian on $\C^d$
 identified with $\R ^{2d}$. Then bigraded solid harmonics of
 bidegree $(m,n)$ are harmonic functions on $\C^d$ of the form
 $$P(z)= \sum _{|\alpha|=m} \sum _{|\beta|= n} c_{\alpha, \beta} z^{\alpha}\overline{z}^{\beta}.$$
 Note that $P(z)$ satisfies the homogeneity condition $P(\lambda
 z)= \lambda ^{m}\overline{\lambda} ^{n}P(z)$ for any $\lambda \in \C
 .$ The elements of $\Hc _{m,n}$ are just restrictions to $\std$ of
 bigraded solid harmonics of bidegree $(m,n)$. As in the case of
 spherical harmonics we have
 $$ L^2 (\std) = \bigoplus _{m,n \geq 0} \Hc _{m,n}.$$
 Here we use the standard inner product
 $$(f,g)_{L^2(\std)}= \int_{\std} f(z')\overline{g(z')}dz'$$
 on the unit sphere. Let $d(m,n)$ stand for the dimension of
 $\Hc_{m,n}$. We fix an orthonormal basis $\{ Y_{m,n} ^j : j =
 1,2,\ldots , d(m,n)\}$ for $\Hc _{m,n}$.\\

 We write $z=x+iy$ for elements of $\C^d$ and identify $z$ with
$(x,y)\in \R ^{2d}$. We use the notation $\N^x$ and $\N^y$ for the
gradient on $\R^d$ in the $x$ and $y$ variables respectively. Let
$\N = ( \N^x , \N^y )$ stand for the gradient on $\R ^{2d}$.
Define the complex gradient $ \N ^z = \frac{1}{2}(\N^x - i \N^y)
$. Then it is clear that $\N^z = ( \frac{\pa}{\pa z_1},
\frac{\pa}{\pa z_2}, \cdots , \frac{\pa}{\pa z_d}) $ where
$\frac{\pa}{\pa z_j} =\frac{1}{2}\Big (\frac{\pa}{\pa x_j}- i
\frac{\pa}{\pa y_j} \Big )$ , $j=1,2, \ldots , d $. If we write
$(x,y)= r(\xi , \eta ) $, $ (\xi , \eta )\in \std $ , then the
gradient $ \N = (\N ^x, \N ^y )$ can be written as
\begin{eqnarray}
\N = \frac{1}{r} \N_0 + (\xi , \eta ) \frac{\pa}{\pa r}
\end{eqnarray}
Here $ \N _0$ is the spherical part of the gradient $\N $. By
writing $\N _0 = (\N_0 ^x, \N_0 ^y) $ we obtain the decomposition
\begin{eqnarray}
\N ^x = \frac{1}{r} \N_0 ^x + \xi \frac{\pa}{\pa r}, \; \;   \N ^y
= \frac{1}{r} \N_0 ^y + \eta \frac{\pa}{\pa r}.
\end{eqnarray}
This leads us to the decomposition
\begin{eqnarray}
\N ^z = \frac{1}{r} \N_0 ^z + \frac{1}{2r} \; \overline{z} \;
\frac{\pa}{\pa r} \label{C}
\end{eqnarray}
where $\N_0 ^z = \frac{1}{2} (\N_0 ^x -i \N_0 ^y) $. With these
notations we can now prove the following result.

\begin{prop}
Let $P_{m,n} $ be any bigraded solid harmonic of bidegree $(m,n)
$. With the notation $\La z,\; w \Ra = z\cdot \overline{w} =
\sum_{j=1}^d z_j \overline{w_j} $ we have the following:

\begin{enumerate}
\item

     \begin{enumerate}
     \item $\La \overline{z}, \; \N^z P_{m,n}(z)\Ra = m
     \overline{P_{m,n}}(z)$ \\
     \item $\La \overline{z}, \; \N_0 ^z P_{m,n}(z) \Ra = \frac{r}{2}(m-n) \overline{P_{m,n}}(z) $\\
     \item $\La \overline{\zeta}, \; \N_0 ^z P_{m,n}(\zeta) \Ra = \frac{i}{2}(\xi \cdot \N_0 ^y \overline{P_{m,n}(\zeta)}
      - \eta \cdot \N_0 ^x \overline{P_{m,n}(\zeta)} )$ where $\zeta = \xi + i \eta .$ \\
     \end{enumerate}

 \item $ \La \N^z P_{m,n},\; \N^z P_{m',n'}\Ra = \frac{1}{r^2}
 \La \N_0 ^z P_{m,n},\; \N_0 ^z P_{m',n'}\Ra +
 \frac{1}{4r^2}((3m+n)m'+(m-n)n')P_{m,n}\overline{P_{m',n'}}.$ \\

 \item For any two functions $f$ and $g$ on $\std $ then\\
   $\La \N_0 ^z f(\zeta), \; \N_0 ^z g(\zeta)\Ra = \frac{1}{4} \N_0 f(\zeta)\cdot \N_0 \overline{g(\zeta)} +
     \frac{i}{4}(\N_0 ^x f(\zeta)\cdot \N_0 ^y \overline{g(\zeta)} - \N_0 ^y f(\zeta)\cdot \N_0 ^x
 \overline{g(\zeta)}).$\\

\item    $ \int _{\std} \N_0 ^x f(\zeta)\cdot \N_0 ^y
\overline{g(\zeta)}d\zeta$

$= - \int _{\std}(\N_0 ^y \cdot \N_0 ^x f(\zeta))
\overline{g(\zeta)} d\zeta +(2d-1) \int _{\std} (\eta \cdot \N_0^x
f(\zeta))\overline{g(\zeta)}d\zeta
   .$ \\

\item $\int _{\std} \La \N_0 ^z P_{m,n}(\zeta), \; \N_0 ^z
\overline{P_{m',n'}(\zeta)} \Ra d\zeta = \lambda _d(m,n)
 \La P_{m,n},\; P_{m',n'} \Ra $ \\
where $\lambda _d(m,n)= \frac{1}{4}((m+n)^2+(4d-3)m-n ). $

 \end{enumerate}

\end{prop}

\begin{proof} 1(a) follows by simple calculation:
$$\La  \overline{z}, \; \N^z P_{m,n} \Ra = \sum_{j=1}^d \overline{z_j}\; \overline{\frac{\pa}{\pa z_j} P_{m,n}(z)}.$$
If $ P_{m,n} = \sum_{|\ap| = m} ^{} \sum_{| \bt |= n} ^{} c_{\ap
\bt}\; z^{\ap}  \overline{z}^{\bt} \; $ then
$$ \; \overline{z_j}
\; \overline{\frac{\pa}{\pa z_j} P_{m,n}(z)} = \sum_{|\ap| =
m}\sum_{|\bt| = n} \overline{c}_{\ap \bt} \; \ap _j \;
\overline{z}^{\ap} z^{\bt}.$$ Summing over $j$ and noting that
$|\ap| = m $, we get the result. To prove $1(b)$ we use
\eqref{C} which gives $\N_0 ^z = r \N^z -
\frac{1}{2}\overline{z}\frac{\pa}{\pa r} $. Hence
$$\N_0 ^z P_{m,n}(z) = r \N^z P_{m,n}(z)- \frac{1}{2} \; \overline{z} \; \frac{\pa}{\pa r}P_{m,n}(z).$$
By virtue of homogeneity $\frac{\pa}{\pa r}P_{m,n}(z) =
\frac{1}{r} (m+n) P_{m,n}(z)$ and therefore
$$\N_0 ^z P_{m,n}(z) = r \N ^z  P_{m,n}(z) -
\frac{(m+n)}{2r} \; \overline{z} \; P_{m,n}(z).$$ Taking inner
product with $\overline{z} $ and using $1(a)$ we get
$$\La \overline{z},\; \N_0 ^z P_{m,n}(z) \Ra = rp \overline{P_{m,n}(z)} - \frac{r}{2}(m+n) \overline{P_{m,n}(z)}= \frac{r}{2}(m-n)\overline{P_{m,n}(z)}.$$
To prove 1(c) we make use of the fact that $(\xi, \eta)\cdot \N
P_{m,n}(\zeta) = 0$ (see Lemma 2.2 in \cite{PX}).
$$\La \overline{\zeta}, \; \N_0 ^z P_{m,n}(z) \Ra = \frac{1}{2}\La (\xi -i \eta), (\N_0 ^x -i \N_0 ^y )P_{m,n} (\zeta)  \Ra .$$
Simplifying and using $(\xi, \eta)\cdot \N P_{m,n}(\zeta) = 0 $ we
obtain
$$\La \overline{\zeta},\; \N_0 ^z P_{m,n}(\zeta) \Ra = \frac{i}{2} ( \xi\cdot \N_0 ^y \overline{P}_{m,n}(\zeta) - \eta \cdot \N_0 ^x
\overline{P}_{m,n}(\zeta)).$$ Coming to the proof of $(2)$ we see
that
\begin{eqnarray*}
  \N ^z P_{m,n}(z) &=& \frac{1}{r} \Big ( \N_0 ^z + \frac{1}{2}\;  \overline{z} \;
\frac{\pa}{\pa r} \Big ) P_{m,n}(z) \\
   &=& \frac{1}{r} \Big ( \N_0 ^z P_{m,n}(z) +
\frac{(m+n)}{2r} \; \overline{z} \; P_{m,n}(z)\Big ).
\end{eqnarray*}
We have a similar expression for $\N ^z P_{m',n'}(z) $.
Consequently,\\ $\La \N ^z P_{m,n}(z), \; \N ^z P_{m',n'}(z) \Ra $
is given by
  $$\frac{1}{r^2} \Big ( \La \N_0 ^z P_{m,n}(z), \; \N_0 ^z P_{m',n'}(z) \Ra + \frac{(m+n)}{2r} \; P_{m,n}(z)  \La \overline{z}, \; \N_0 ^z P_{m',n'}(z)
  \Ra$$
  $$   + \frac{(m'+n')}{2r} \; \overline{P}_{m',n'}(z)  \La  \N_0 ^z P_{m,n}(z), \; \overline{z}\Ra
  + \frac{(m+n)(m'+n')}{4r^2} \; r^2 \; P_{m,n}(z)\overline{P}_{m',n'}(z)\Big ).$$
 Using $1(b)$ and simplifying we get\\

$\La \N ^z P_{m,n}(z), \;  \N ^z P_{m',n'}(z) \Ra$\\

$ = \frac{1}{r^2} \; \La \N_0 ^z P_{m,n}(z), \; \N_0 ^z
P_{m',n'}(z) \Ra
+\frac{1}{4r^2}((3m+n)m'+(m-n)n')\;P_{m,n}(z)\overline{P}_{m',n'}(z).$\\

\hspace{-5mm} Recalling the definition of $ \N_0 ^z $ we see that
$$\La \N_0 ^z f(\zeta), \; \N_0 ^z g(\zeta) \Ra = \frac{1}{4} \La (\N_0 ^x -i \N_0 ^y)f(\zeta), \; (\N_0 ^x -i \N_0 ^y)g(\zeta)\Ra .$$
The right hand side simplifies into
$$ \N_0 f(\zeta)\cdot \N_0 \overline{g}(\zeta)+i ( \N_0 ^x f(\zeta)\cdot \N_0 ^y \overline{g}(\zeta)- \N_0 ^y f(\zeta)\cdot \N_0 ^x \overline{g}(\zeta)) $$
which is the required expression for proving (3). Considering now
the integral in $(4)$ which can be evaluated using Proposition
8.7, Chapter 1 in \cite{DX}. Writing $\N_0 = (D_1, \cdots , D_d,
D_{d+1}, \cdots , D_{2d})$ so that $\N_0 ^x =(D_1,D_2, \cdots ,
D_d)$, $ \N_0 ^y = (D_{d+1}, \cdots , D_{2d}) $ we have
\begin{eqnarray*}
 & & \int _{\std}\N_0 ^x f(\zeta)\cdot \N_0 ^y \overline{g}(\zeta)d
 \zeta  =  \sum _{j=1} ^d \int _{\std}D_jf(\zeta)D_{d+j}\overline{g}(\zeta)d
 \zeta \\
 & =& -\sum_{j=1}^d \int_{\std}D_{d+j}D_jf(\zeta)\overline{g}(\zeta) d\zeta +
 (2d-1)\sum_{j=1}^d\int_{\std} \eta _j D_jf(\zeta) \overline{g}(\zeta)d\zeta \\
 & =& -\int_{\std}(\N_0 ^y\cdot
\N_0^xf(\zeta))\overline{g}(\zeta)d\zeta
 +(2d-1)\int _{\std}(\eta \cdot \N_0^xf(\zeta))
 \overline{g}(\zeta)d\zeta.
\end{eqnarray*}
This proves $(4)$. Finally in order to prove $(5)$ we integrate
$(3)$ over $\std $ and make use of $(4)$. The result is , with
$P=P_{m,n},\; Q=P_{m',n'}$
\begin{eqnarray*}
 \int_{\std}\La \N_0 ^zP(\zeta), \; \N_0 ^z Q(\zeta) \Ra d\zeta=\frac{1}{4}\int_{\std}\N_0P(\zeta) \cdot\N_0 \overline{Q}(\zeta)(\zeta)d\zeta & &\\
    +  \frac{i}{4}(2d-1)\int_{\std}(\eta \cdot \N_0 ^x P(\zeta) - \xi \cdot \N_0
 ^yP(\zeta))\overline{Q}(\zeta)d\zeta
\end{eqnarray*}
 since $\N_0 ^x\cdot \N_0 ^y f = \N_0 ^y \cdot \N_0 ^x f $. Using
 $1(c)$ we convert the second integral in the above into
 $ -\frac{2d-1}{2} \int_{\std} \La \overline{\zeta}, \; \N_0 ^z \overline{P}(\zeta)\Ra \; \overline{Q(\zeta)} \; d\zeta.$
 Since $P= P_{m,n} $ it follows that $\overline{P} $ is of bigraded $(n,m)$ and hence by
 $1(b)$ we have $\La \overline{\zeta}, \; \N_0 ^z \overline{P}(\zeta) \Ra = \frac{(n-m)}{2}P(\zeta)$. Using this we have
 obtained
 \begin{eqnarray*}
 & &\int _{\std} \La \N_0 ^zP, \; \N_0 ^z Q \Ra \; d\zeta\\
 & & =\frac{1}{4}\int_{\std}\N_0P(\zeta)\cdot \N_0 \overline{Q}(\zeta) \; d\zeta
  +\frac{(2d-1)}{4}(m-n)\int_{\std}P(\zeta)\overline{Q}(\zeta)
  d\zeta.
 \end{eqnarray*}
 The first integral simplifies to
 $$-\int_{\std}\Delta_0P(\zeta)\overline{Q}(\zeta)d\zeta
 = (m+n)(m+n+2d-2)\int_{\std}P(\zeta) \; \overline{Q}(\zeta) \; d\zeta $$
 where $\Delta_0 $ is the spherical Laplacian for which $P $ is an
 eigenfunction. Thus
 \begin{eqnarray*}
  \int _{\std} \La \N_0 ^zP, \; \N_0 ^z Q \Ra \; d\zeta = \frac{1}{4}((m+n)(m+n+2d-2)+(2d-1)(m-n))& &\\
  \int_{\std}(\N_0P \cdot \N_0 \overline{Q})(\zeta) \; d\zeta
 \end{eqnarray*}
 which can be simplified to prove $(5)$.
 \end{proof}
\subsection{Mixed norm estimates for $S_j$.}
In this subsection we prove the mixed norms estimates for $S_j $
stated in Theorem 1.2. As in Section 2.2 we use the idea of Rubio
de Francia along with some weighted norm inequalities satisfied by
$S_j$ and the transform properties of the vector $S=(S_1, S_2,
\cdots , S_d )$ under the action of the unitary group. As the
ideas can be applied even to higher order Riesz transforms we
start with the following definition.\\

Given a bigraded solid harmonic $ P\in \Hc _{m,n}$ we define $R_P
$ the higher order Riesz transform associated to $P $ by the
prescription
$$ W(R_Pf)=W(f)G(P)H^{-\frac{1}{2}(m+n)} .$$
\noindent Here $W(f) $ is the Weyl transform of $f$ and $G(P)$ is
the Weyl correspondence of the polynomial $P$. We refer to Geller
\cite{G}, Thangavelu \cite{ST1} and Sanjay-Thangavelu \cite{SST},
for various facts about these objects. We let $\rho
(k)f(z)=f(k\cdot z) $ stand for the action of $U(d) $ on
functions. We also make use of the metaplectic representation $\mu
(k) $ which is defined by the property
$$W (\rho (k)f) = \mu(k)^* W(f)\mu(k) .$$
The operator $\mu(k) $, $k \in U(d) $ are unitary on $L^2(\R^d)$.
The following proposition is easy to prove.
\begin{prop}
 For any $k\in U(d)$ we have
 $$ \rho (k)R_Pf(z) = R_{\rho (k)P}(\rho (k)f)(z).$$
\end{prop}
\begin{proof}
 Indeed, we notice that
 \begin{eqnarray*}
   W(R_{\rho(k^{-1})P}f) & = & W(f)G(\rho (k^{-1})P)H^{-\frac{1}{2}(m+n)}  \\
    &=& W(f)\mu (k) G(P)\mu (k)^* H^{-\frac{1}{2}(m+n)}.
 \end{eqnarray*}
 Since $\mu (k)^* $ commutes with $H^{-\frac{1}{2}(m+n)} $ we can
 rewrite the above as

 $$W(R_{\rho (k^{-1})P} f ) = \mu (k) W(\rho (k)f)G(P) H^{-\frac{1}{2}(m+n)}\mu (k)^* .$$
 This simply means that
 $$W(\rho(k)R_{\rho (k^{-1})P} f ) = W(\rho (k)f)G(P) H^{-\frac{1}{2}(m+n)} $$
 which proves the proposition.
\end{proof}
 In \cite{SST} it has been shown that the operators $R_P $ are Oscillatory
 singular integral operators. Appealing to the theorem of
\cite{LZ} we obtain
 \begin{thm}
 For any $P \in \Hc _{m,n}$ the operators $R_P$ satisfy the weighted
 norm inequality
 $$\Big ( \int_{\C^d}|R_{P}f(z)|^p w(z)dz \Big )^{\frac{1}{p}}\leq C \Big ( \int_{\C^d}|f(z)|^p w(z)dz \Big )^{\frac{1}{p}}  $$
 for $1<p<\infty$ whenever $w\in A_p(\C^d)$.
 \end{thm}

 Let us specialize to the the case where $P_w(z) = \sum_{j=1}^d z_j \overline{w}_j $,
 $w \in \std $ so that $ R_{P_w} = \sum_{j=1}^d \overline{w}_j
 S_j $. In view of the proposition we get the identity
 $$R_{P_w}f(k\cdot z) = R_{P_{k\cdot w}}(\rho (k)f)(z). $$
 which is the same as saying
 $$\sum_{j=1}^d \overline{w}_jS_j f(k\cdot z)= \sum_{j=1}^d (k\cdot w)_jS_j (\rho (k)f)(z) $$
 We can now use this identify along with the above weighted norm
 inequality to use the idea of Rubio de Francia to get the
 required inequality. We need to remember that radial weight
 function $ w \in A_p(\C^d)$ if and only if $ w(r) \in A_p
 ^{d-1}(\R ^+).$ By taking $ w =e_j$, the coordinate vectors we get
 Theorem 1.2.\\

 In the next subsection we apply Theorem 1.2 to get some vector
 valued inequality for Laguerre Riesz transforms which are more refined than those proved in
 Theorem 2.6.

\subsection{The vector of Riesz transforms}
 Recall that the Riesz transforms associated to the special Hermite operator $L$ are given by $S_j =
 Z_jL^{-\frac{1}{2}}$ and $\overline{S}_j =  \overline{Z}_j L^{- \frac{1}{2}}
 $, $ j=1,2,\dots ,d $ with $Z_j = \frac{\pa}{\pa z_j}+ \frac{1}{4}\overline{z}_j $ and
 $\overline{Z}_j = \frac{\pa}{\pa \overline{z}_j}- \frac{1}{4}z_j $.
 Let $Sf = (S_1, \cdots, S_d)f $ and $\overline{S}f = (\overline{S}_1, \cdots, \overline{S}_d)f $ stand for the vectors of Riesz transforms;
 note that $Sf = \Big (\N ^z + \frac{1}{4} \overline{z}\Big ) L^{-\frac{1}{2}}f$ and
 $\overline{S}f = \Big (\overline{\N ^z} - \frac{1}{4} z \Big ) L^{-\frac{1}{2}}f$. We first calculate the spherical harmonic coefficients
 of $Sf $ and $\overline{S}f $.\\

 Let $ e^{-tL}$ stand for the special Hermite semigroup which is
 given by twisted convolution with
 $$p_t(z) = c_d (\sinh t)^{-d} e^{-\frac{1}{4}(\coth t)|z|^2} $$
 That is to say,
 $$ e^{-tL}f(z) = f \times p_t(z)= \int_{\C^d}f(z-w)p_t(w) e^{\frac{i}{2} \Im(z \cdot \overline{w})} dw. $$
 We can express $ L^{-\frac{1}{2}}f(z) $ in terms of $ e^{-tL}f $ in the usual way
 $$ L^{-\frac{1}{2}} f(z)= \frac{1}{\sqrt{\pi}}\int_0 ^{\infty } e^{-tL}f(z)t^{-\frac{1}{2}}dt. $$
 In order to find the spherical harmonic expansion of $L^{-\frac{1}{2}} f(z)$ we
 calculate the coefficients of $ f \times p_t(z)$. For $\delta > -\frac{1}{2} $ let us define
 $$ \varphi_k ^{\delta}(r) = \Big ( \frac{\Gamma (k+1) 2^{-\delta}}{\Gamma (k+\delta + 1)}\Big )^{\frac{1}{2}} L_k ^{\delta}(\frac{1}{2}r^2)e^{-\frac{1}{4}r^2} $$
 so that $ \{ \varphi_k ^{\delta}: k=0,1,2, \cdots \}$ forms an orthonormal basis for $ L^2(\R ^+ , r^{2\delta +1 }dr ) $. We set
 $$ R_k ^{\delta } (g) =  \frac{ \Gamma (k+1) 2^{-\delta }}{\Gamma (k+ \delta + 1)} \int _0 ^{\infty}
 g(r) L_k ^{\delta }(\frac{1}{2}r^2)e^{-\frac{1}{4}r^2} r^{2\delta +1}dr. $$
 We fix an orthonormal basis $\{ Y_{m,n} ^j : j= 1,2,\ldots , d(m,n), \; m,n \in \mathbb{N} \} $ for $L^2 (\std) $ consisting of bigraded
 spherical harmonics. Let
 $$k_t ^\delta (r,s) = \sum_{k=0} ^\infty e^{-(2k+\delta+1)t }\varphi_k ^{\delta}(r) \varphi_k ^{\delta}(s).  $$
 which can be expressed in terms of the kernel $K_t ^\delta (r,s) $ introduced in
 Subsection 2.2. Indeed, $k_t ^\delta (r,s)= 2^{-\delta-1} K_{\frac{t}{2}}^{\delta} (\frac{r}{\sqrt{2}},\; \frac{s}{\sqrt{2}} ) $. Let $f_{m,n} ^j $
 stand for the spherical harmonic coefficients of $f $.
 \begin{prop}
 For each $m,n\in \mathbb{N}, \; j=1,2,\dots , d(m,n) $ we have
 $$\int _{\std} f \times p_t(r\zeta) \overline{Y_{m,n} ^j(\zeta)} d\zeta = \widetilde{T}_t ^{d+m+n-1}f_{m,n} ^j(r) Y_{m,n} ^j (\zeta) $$
 where
 $$\widetilde{T}_t ^{d+m+n-1} g(r) = e^{-t(m-n)} \int _0 ^{\infty} g(s) (rs)^{m+n}k_t ^{d+m+n-1}(r,s) s^{2d-1}ds .$$

 \end{prop}
 The proposition follows immediately from the Hecke - Bochner
 formula for the special Hermite projections ( see Theorem 2.6.1 or Equation 2.6.10  in \cite{ST1}).
 According to this formula, when $ f(r\zeta)= g(r)Y(\zeta) $ with $Y \in \Hc_{m,n}  $ one has
 $$f \times \varphi _k (r\zeta)= c_dR_{k-m} ^{d+m+n-1}(\widetilde{g})\varphi_k ^{d+m+n-1}(r) Y(\zeta)r^{m+n} $$
 where $\widetilde{g}(r)=r^{-(m+n)}g(r) $. This shows that, for $f $ as above,
 $$e^{-tL}f(r\zeta) = \sum_{k=m} ^{\infty} e^{-(2k+d)t}R_{k-m}^{d+m+n-1}(\widetilde{g}) \varphi _{k-m}^{d+m+n-1}(r)r^{(m+n)}Y(\zeta) $$
 which simplifies to
 $$e^{-tL}f(r\zeta)= e^{-t(m-n)}\Big ( \int_0 ^{\infty}g(s)(rs)^{m+n}k_t ^{d+m+n-1}(r,s)ds \Big ) r^{m+n} Y(\zeta). $$
 Expanding $f $ in terms of $ Y_{m,n} ^j $ and using the above calculations we
 get the proposition.
 \begin{rem}
  As a by-product of the above calculation we get the interesting
  formula
  \begin{eqnarray*}
   \int _{\std} p_t(rz'-sw')e^{-\frac{i}{2}rs (z'\cdot w')}\overline{Y_{m,n} ^j(w')}dw'& &\\
     = Y_{m,n} ^j(z')(rs)^{m+n}k_t ^{d+m+n-1}(r,s)e^{-t(m-n)}.& &
 \end{eqnarray*}
 \end{rem}
 \begin{cor}
 If $F_{m,n}^j(r) $ stand for the spherical harmonic coefficients of $L^{-\frac{1}{2}}f $  then
 we have
 $$F_{m,n} ^j (r) = \frac{1}{\sqrt{\pi}}\int_0 ^{\infty} \int_0 ^{\infty} e^{-t(m-n)}(rs)^{m+n}k_t ^{d+m+n-1}(r,s)f_{m,n} ^j(s)s^{2d-1}t^{-\frac{1}{2}}ds dt $$
 \end{cor}
 Thus we have obtained the following expansion,
  $$ L^{-\frac{1}{2}} f(r\zeta)= \sum_{m,n=0}^{\infty}\sum_{j=1}^{d(m,n)}F_{m,n} ^j(r)Y_{m,n} ^j(\zeta). $$
  In order to find the spherical harmonic coefficients of $Sf $ and $\overline{S}f $, let
  us calculate the action of $\Big (\N^z + \frac{1}{4} \overline{z}\Big ) $ and $\Big (\overline{\N ^z} - \frac{1}{4} z \Big ) $ on functions of the form $F(r)Y(\zeta) $
  where $Y \in \Hc _{m,n} $. In view of \eqref{C} we have
  $$  \N ^z (FY))(z)= \frac{1}{r}F(r)\N_0 ^zY(\zeta) + \frac{1}{2r}\frac{\pa F}{\pa r}(r)\; \overline{z}, $$
  $$ \overline{\N ^z }(FY))(z)= \frac{1}{r}F(r)\overline{\N_0 ^z}Y(\zeta) + \frac{1}{2r}\frac{\pa F}{\pa r}(r)\; z   $$
  and consequently
  $$ \Big ( \N ^z +\frac{1}{4}\overline{z}\Big )(FY)(r\zeta)= \frac{1}{2} \Big ( \frac{\pa}{\pa r}+\frac{1}{2}r \Big )F(r)Y(\zeta)
  \overline{\zeta}+\frac{1}{r}F(r)\N_0 ^zY(\zeta),$$

  $$ \Big ( \overline{\N ^z} - \frac{1}{4}z \Big )(FY)(r\zeta)= \frac{1}{2} \Big ( \frac{\pa}{\pa r}-\frac{1}{2}r \Big )F(r)Y(\zeta)
  \zeta + \frac{1}{r}F(r)\overline{\N_0 ^z}Y(\zeta).$$
  Using this we can easily prove the following proposition on
  the vectors ( $Sf $ and $ \overline{S}f $ ) of Riesz transforms.\\
  \begin{prop}
  The square of the $L^2(\std) $ norm of $\La Sf, \; Sf \Ra + \La \overline{S}f, \; \overline{S}f \Ra $ is the sum of the following
  five terms:
  \begin{enumerate}
  \item $\sum_{m,n=0}^{\infty} \sum_{j=1}^{d(m,n)}\Big | \frac{1}{2} \Big (\frac{\pa}{\pa r}+ \frac{1}{2}r \Big ) F_{m,n}^j(r)\Big | ^2 = A_1 (r)^2,$\\
  \item $ \sum_{m,n=0}^{\infty} \sum_{j=1}^{d(m,n)}\Big | \frac{1}{2} \Big (\frac{\pa}{\pa r} - \frac{1}{2}r \Big ) F_{m,n}^j(r)\Big | ^2= A_2 (r)^2,$\\
  \item $\sum_{m,n=0}^{\infty} \sum_{j=1}^{d(m,n)}\frac{1}{r^2} \lambda _d(m,n)|F_{m,n}^j(r)|^2 = A_3(r)^2, $\\
  \item $\sum_{m,n=0}^{\infty} \sum_{j=1}^{d(m,n)}\frac{1}{r^2} \lambda _d(n, m)|F_{m,n}^j(r)|^2 = A_4(r)^2, $\\
  \item $ \sum_{m,n=0}^{\infty} \sum_{j=1}^{d(m,n)} \frac{(m-n)}{2}  |F_{m,n}^j(r)|^2 = A_5(r).$

  \end{enumerate}
  \end{prop}
  \begin{proof}
  In view of the above expression for $\Big ( \N^z + \frac{1}{4} \overline{z} \Big ) (FY)$ we see that
  $$ \La \Big ( \N ^z + \frac{1}{4} \overline{z} \Big )(F_{m,n}^{j}Y_{m,n}^{j})(r\zeta), \; \Big ( \N^z + \frac{1}{4} \overline{z} \Big )
  (F_{m',n'}^{j'}Y_{m',n'}^{j'})(r\zeta)  \Ra $$
  is the sum of the three terms
  $$\frac{1}{2} \Big ( \frac{\pa}{\pa r} +\frac{1}{2}r \Big ) F_{m,n}^j(r)\frac{1}{2} \Big ( \frac{\pa}{\pa r}
  +\frac{1}{2}r\Big )\overline{ F_{m',n'}^{j'}}(r)Y_{m,n}^j(\zeta)\overline{Y_{m',n'}^{j'}(\zeta)}, $$
  $$\frac{1}{r^2}F_{m,n}^j(r)\; \overline{F_{m',n'}^{j'}}(r)\La \N_0 ^z Y_{m,n} ^j,\; \N_0 ^z Y_{m',n'} ^{j'}\Ra ,$$
  and
  $$ 2 \Re \frac{1}{2} \Big ( \frac{\pa}{\pa r} + \frac{1}{2} r \Big )F_{m,n}^j(r)Y_{m,n}^j(\zeta) \La \overline{\zeta},\; \N_0 ^z Y_{m',n'} ^{j'} \Ra
  \; \frac{1}{r}\; \overline{F_{m',n'}^{j'}}(r).$$
  Integrating over $\std $ and making use of Proposition 3.1 we
  can see that $\La Sf, \; Sf \Ra $ is the sum of the terms (1), (3) and
  \begin{eqnarray}
   2 \Re \sum_{m,n=0}^{\infty} \sum_{j=1}^{d(m,n)} \frac{1}{2} \Big  (\frac{\pa}{\pa r}+ \frac{1}{2}r\Big ) F_{m,n}^j(r)
  \;\frac{(m-n)}{2r}\; \overline{F_{m,n}^j(r)}. \label{D}
  \end{eqnarray}

  And similarly $\La \overline{S}f, \; \overline{S}f \Ra $ is the
  sum of (2), (4) and
\begin{eqnarray}
2 \Re \sum_{m,n=0}^{\infty} \sum_{j=1}^{d(m,n)} \frac{1}{2} \Big
(\frac{\pa}{\pa r}- \frac{1}{2}r\Big ) F_{m,n}^j(r)
  \;\frac{(n-m)}{2r}\; \overline{F_{m,n}^j(r)}.
\end{eqnarray}
By adding $\La Sf, \; Sf \Ra $ with $\La \overline{S}f, \;
\overline{S}f \Ra $, we will get the required result.

  \end{proof}
  \subsection{Revisiting Laguerre Riesz transforms: }

 In the course of the proof of Proposition 3.7 we have shown that $\La Sf, \; Sf \Ra $
 is the sum of the terms (1), (3) and \eqref{D} of Proposition 3.7. By Cauchy-Schwarz
 inequality the third term viz. \eqref{D} is dominated by the sum of first
 two terms and hence the mixed norm estimates
 $$\|S_jf \|_{L^{p,2}(\C^d, w)}\leq  C \| f \|_{L^{p,2}(\C^d, w)} $$
 will follow once we prove the estimate
 $$\Big ( \int_0 ^{\infty} A_j(r)^p w(r)r^{2d-1}dr \Big )^{\frac{1}{p}}\leq C \| f \|_{L^{p,2}(\C^d, w)} $$
 for $j=1,2 $. On the other hand from Theorem 1.2 and Proposition 3.7 we
 should be able to deduce the above inequality for all $j=1,2,\ldots,5 $. Since
 the term $A_5 $ is not nonnegative, this expectation may be false. What
 we can deduce from Theorem 1.2 is the following.\\

 Let $ L^p_h(\C^d) $ stand for the subspace of $L^p(\C^d) $ consisting of $f$ for which
 $$\int_{\std}f(r\zeta)Y^j_{m,n}(\zeta)d\zeta = 0 $$
 for all $m<n$. Similarly we can define $L^p_{ah}(\C^d) $ with the condition
 $$\int_{\std}f(r\zeta)Y^j_{m,n}(\zeta)d\zeta = 0 $$
 for all $m\geq n $. Clearly, $L^p_h(\C^d)\oplus L^p_{ah}(\C^d) = L^p(\C^d)$.
 For $f \in L^p_h(\C^d) $, all the terms appearing in Proposition 3.7, including $A_5 $
 are nonnegative. Hence we get the following result.
 \begin{thm}
 For $f \in L^p_h(\C^d) $, $1<p<\infty $ and $w\in A_p ^{d-1}(\R^+) $ we have the inequalities
$$\Big ( \int_0 ^{\infty} A_j(r)^p w(r)r^{2d-1}dr \Big )^{\frac{1}{p}}\leq C \| f \|_{L^{p,2}(\C^d, w)}$$
 for $j=1,2, \ldots, 5 $.
 \end{thm}
 It is also possible to deduce the above result directly as a
 consequence of the result for the Hermite Riesz transforms. For
 example, consider the inequality
 $$\int_0 ^{\infty}\Big (\sum_{m\geq n}^{\infty} \sum_{j=1}^{d(m,n)} \frac{(m+n)^2}{r^2}|F_{m,n}^j(r)|^2\Big )^{\frac{p}{2}}w(r)r^{2d-1}dr$$
 $$ \leq C \int_0 ^{\infty}\Big (\sum_{m\geq n}^{\infty} \sum_{j=1}^{d(m,n)}|f_{m,n}^j(r)|^2\Big )^{\frac{p}{2}}w(r)r^{2d-1}dr .$$
 for $f \in  L^p_h(\C^d)$. Recall that
 $$F_{m,n} ^j (r) = \frac{1}{\sqrt{\pi}}\int_0 ^{\infty} \int_0 ^{\infty} e^{-t(m-n)}(rs)^{m+n}k_t ^{d+m+n-1}(r,s)f_{m,n} ^j(s)s^{2d-1}t^{-\frac{1}{2}}dtds. $$
 When $ m\geq n $ the above is dominated by
 $$|F_{m,n} ^j (r)| \leq \frac{1}{\sqrt{\pi}}\int_0 ^{\infty} \int_0 ^{\infty} (rs)^{N}k_t ^{d+N-1}(r,s)|f_{m,n} ^j(s)|s^{2d-1}t^{-\frac{1}{2}}dtds $$
 where $ m+n=N$. The terms on the right hand side are precisely the
 corresponding terms in the study of Riesz transforms on $\R ^{2d} $.
 Then we can appeal to \eqref{B} to get the desired inequality.\\

 Similarly, we can prove the inequality
$$\int_0 ^{\infty}\Big ( \sum_{m\geq n}^{\infty} \sum_{j=1}^{d(m,n)}\Big | \frac{1}{2} \Big (\frac{\pa}{\pa r}\pm \frac{1}{2}r \Big ) F_{m,n}^j(r)\Big | ^2\Big )^{\frac{p}{2}}w(r)
  r^{2d-1}dr$$ $$ \leq C  \int_0 ^{\infty}\Big (\sum_{m\geq n}^{\infty} \sum_{j=1}^{d(m,n)}|f_{m,n}^j(r)|^2\Big )^{\frac{p}{2}}w(r)r^{2d-1}dr.$$
 This result will be proved once we show that the operators taking
 $ f_{m,n}^j $ into $\Big (\frac{\pa}{\pa r}\pm\frac{1}{2}r \Big )
 F_{m,n}^j(r) $ are singular integral operators on the homogeneous
 space $(\R ^+, d\mu _{d-1}) $ whose kernels satisfy uniform estimates. But this is
 easy to see: the kernels of these operators are given by
 $$\frac{1}{2} \Big (\frac{\pa}{\pa r}\pm \frac{1}{2}r \Big )\int_0 ^{\infty} e^{-t(m-n)}(rs)^{m+n}k_t ^{d+m+n-1}(r,s)t^{-\frac{1}{2}} dt $$
 Since $ k_t ^{d+m+n-1}(r,s) = 2^{-d-m-n} K_{\frac{t}{2}}^{d+m+n-1} (\frac{r}{\sqrt{2}},\; \frac{s}{\sqrt{2}} )$ it is enough to estimate the kernels
 $$\Big (\frac{\pa}{\pa r}\pm r \Big )\int_0 ^{\infty} e^{-t(m-n)}(rs)^{m+n}K_t ^{d+m+n-1}(r,s)t^{-\frac{1}{2}} dt $$
 Since we are assuming $m \geq n $, the estimation of this is similar to
 that of
 $$ \Big (\frac{\pa}{\pa r}\pm r \Big )\int_0 ^{\infty} (rs)^{N}K_t ^{N+d-1}(r,s)t^{-\frac{1}{2}} dt $$
 which has been already done in Section 2. This proves the
 required estimates and hence the result.

\begin{center}
{\bf Acknowledgments}
\end{center}

We are very thankful to G. Garrigos for pointing out the method of
Rubio de Francia. The first author is thankful to CSIR, India, for
the financial support. The work of the second author is supported
by J. C. Bose Fellowship from the Department of Science and
Technology (DST) and also by a grant from UGC via DSA-SAP.


\end{document}